\documentclass[12pt]{amsart}
\usepackage{amsfonts,amsthm,amsmath,amssymb,verbatim}
\usepackage{graphicx}

\newtheorem{thm}{Theorem}[section]

\newtheorem{problem}{Problem}

\theoremstyle{remark}

\newtheorem{example}[thm]{Example}
\newtheorem{remark}[thm]{Remark}
\newtheorem{defin}{Definition}



\def\C{\mathbb{C}}

\def\Z{\mathbb{Z}}
\def\P{\mathbb{P}}

\def\R{\mathbb{R}}

\def\Volume{{\rm Volume}}

\def\wt{\widetilde}

\def\om{\omega}
\def\l{\lambda}

\def\w0{\overline{w_0}}
\def\id{\rm id}

\title[N.--O. polytopes of flag varieties for classical groups]
{Newton--Okounkov polytopes of flag varieties for classical groups}
\author{Valentina Kiritchenko}
\email{vkiritch@hse.ru}
\thanks{The study has been partially funded by the Russian Academic Excellence Project '5-100'.}

\address{Laboratory of Algebraic Geometry and Faculty of Mathematics\\
National Research University Higher School of Economics, Russian Federation\\
Usacheva str. 6, 119048 Moscow, Russia}
\address{Institute for Information Transmission Problems, Moscow, Russia}

\date{}
\dedicatory{To my teacher R.K.Gordin with gratitude and admiration}

\begin{document}
\begin{abstract}
For classical groups $SL_n(\C)$, $SO_n(\C)$ and $Sp_{2n}(\C)$, we define uniformly geometric valuations on the corresponding complete flag varieties.
The valuation in every type comes from a natural coordinate system on the open Schubert cell, and is combinatorially related to the Gelfand--Zetlin pattern in the same type.
In types $A$ and $C$, we identify the corresponding Newton--Okounkov polytopes with the Feigin--Fourier--Littelmann--Vinberg polytopes.
In types $B$ and $D$, we compute low-dimensional examples and formulate open questions.
\end{abstract}

\maketitle

\section{Introduction}
Toric geometry and theory of Newton polytopes exhibited fruitful connections between algebraic geometry and convex geometry.
After the Kouchnirenko and Bernstein--Khovanskii theorems were proved in the 1970-s (for a reminder see Section \ref{ss.toric}), Askold Khovanskii asked how to extend these results to the setting where a complex torus is replaced by an arbitrary connected reductive group.
In particular, he advertised widely the problem of finding the right analogs of Newton polytopes for non-toric varieties such as spherical varieties (classical examples of spherical varieties are reviewed in Section \ref{ss.sphere}).
Notion of Newton polytopes was extended to spherical varieties by Andrei Okounkov in the 1990-s \cite{O97,O98}.
Later, his construction was developed systematically in \cite{KK, LM}, and the resulting theory of Newton--Okounkov convex bodies is now an active field of algebraic geometry.

While Newton--Okounkov convex bodies can be defined for line bundles on arbitrary varieties (without a group action), they are easier to deal with in the case of varieties with an action of a reductive group.
In the latter case, theory of Newton--Okounkov convex bodies is closely related with representation theory.
For instance, Gelfand--Zetlin (GZ) polytopes and Feigin--Fourier--Littelmann--Vinberg (FFLV) polytopes (see Section \ref{s.GZ} for a reminder) arise naturally as Newton--Okounkov polytopes of flag varieties.

\subsection{Newton--Okounkov convex bodies}
\label{ss.toric}

%survey progress made in realizing Khovanskii program for the last 30 years.
In this section, we recall construction of Newton--Okounkov convex bodies for the general mathematical audience.
Let us start from the definition of Newton polytopes.

\begin{defin} Let $f=\sum_{\alpha\in\Z^n} c_\alpha x^\alpha$ be a Laurent polynomial in $n$ variables (here the multiindex notation $x^\alpha$ for $x=(x_1,\ldots,x_n)$ and $\alpha=(\alpha_1,\ldots,\alpha_n)\in\Z^n$ stands for $x_1^{\alpha_1}\cdots x_n^{\alpha_n}$).
The {\em Newton polytope} $\Delta_f\subset\R^n$ is the convex hull of all $\alpha\in \Z^n$ such that $c_\alpha\ne0$.
\end{defin}
By definition, Newton polytope is a lattice polytope, that is, its vertices lie in $\Z^n$.

\begin{example}\label{e.hyp} For $n=2$ and $f=1+2x_1+x_2+3x_1x_2$, the Newton polytope $\Delta_f$ is the square with
the vertices $(0,0)$, $(1,0)$, $(0,1)$ and $(1,1)$.
\end{example}

Note that Laurent polynomials with complex coefficients are well-defined functions at all points $(x_1,\ldots,x_n)\in\C^n$ such that $x_1,\ldots,x_n\ne 0$.
They are regular functions on the complex torus $(\C^*)^n:=\C^n\setminus\bigcup_{i=1}^n\{x_i=0\}$.

\begin{thm}\cite{Kou}\label{t.K}
For a given lattice polytope $\Delta\subset\R^n$, let
$f_1(x_1,\ldots,x_n)$,\ldots, $f_n(x_1,\ldots,x_n)$ be a generic collection of Laurent polynomials with the Newton polytope $\Delta$.
Then the system $f_1=\ldots=f_n=0$ has $n!\Volume(\Delta)$ solutions in the complex torus $(\C^*)^n$.
\end{thm}

The Kouchnirenko theorem can be viewed as a generalization of the classical Bezout theorem.
The Newton polytope serves as a refinement of the degree of a polynomial.
This makes the Kouchnirenko theorem applicable to collections of polynomials which are not generic among all polynomials of given degree but only among polynomials with given Newton polytope.
For instance, the Kouchnirenko theorem applied to a pair of generic polynomials with Newton polytope  as in Example \ref{e.hyp} yields the correct answer $2$ while Bezout theorem yields an incorrect answer $4$ (because of two extraneous solutions at infinity).
A more geometric viewpoint on the Bezout theorem and its extensions stems from enumerative geometry and will be discussed in Section \ref{ss.sphere}.
The Koushnirenko theorem was extended to the systems of Laurent polynomials with distinct Newton polytopes by David Bernstein
and Khovanskii using mixed volumes of polytopes \cite{B75}.
Further generalizations include explicit formulas for the genus and Euler characteristic of complete intersections $\{f_1=0\}\cap\ldots\cap\{f_m=0\}$ in $(\C^*)^n$ for $m<n$ \cite{Kh78}.

We now consider a bit more general situation.
Fix a finite-dimensional vector space $V\subset\C(x_1,\ldots,x_n)$ of rational functions on $\C^n$.
Let $f_1$,\ldots, $f_n$ be a generic collection of functions from $V$, and $X_0\subset\C^n$ an open dense subset obtained by removing poles of these functions.
How many solutions does a system $f_1=\ldots=f_n=0$ have in $X_0$?
For instance, if $V$ is the space spanned by all Laurent polynomials with a given Newton polytope, and $X_0=(\C^*)^n$, then
the answer is given by the Kouchnirenko theorem.
Here is a simple non-toric example from representation theory.

\begin{example}\label{e.flag}
Let $n=3$.
Consider the adjoint representation of $SL_3(\C)$ on the space ${\rm End}(\C^3)$ of all linear operators on $\C^3$.
That is, $g\in SL_3(\C)$ acts on an operator $X\in{\rm End}(\C^3)$ as follows:
$${\rm Ad}(g): X\mapsto gXg^{-1}.$$
Let $U^-\subset SL_3(\C)$ be the subgroup of lower triangular unipotent matrices:
$$U^-=\left\{\begin{pmatrix}1&0&0\\x_1&1&0&\\x_2&x_3&1&\\ \end{pmatrix} \ |  \ (x_1,x_2,x_3)\in \C^3\right\}.$$
To define a subspace $V\subset\C(x_1,x_2,x_3)$ we restrict functions from the dual space ${\rm End}^*(\C^3)$ to the $U^-$-orbit ${\rm Ad}(U^-)E_{13}$ of the operator $E_{13}:=e_1\otimes e_3^*\in {\rm End}(\C^3)$ (here ($e_1$, $e_2$, $e_3$) is the standard basis in $\C^3$).
More precisely, a linear function $f\in {\rm End}^*(\C^3)$ yields the polynomial $\hat f(x_1,x_2,x_3)$ as follows:
$$\hat f(x_1,x_2,x_3):=f\left(\begin{pmatrix}1&0&0\\x_1&1&0&\\x_2&x_3&1&\\ \end{pmatrix} \begin{pmatrix}0&0&1\\0&0&0&\\0&0&0&\\ \end{pmatrix} \begin{pmatrix}1&0&0\\x_1&1&0&\\x_2&x_3&1&\\ \end{pmatrix}^{-1} \right)$$
It is easy to check that the space $V$ is spanned by 8 polynomials: $1$, $x_1$, $x_2$, $x_3$, $x_1x_2-x_1^2x_3$, $x_1x_3$,  $x_2x_3$, $x_2^2-x_1x_2x_3$.
It will be clear from the next section that the Kouchnirenko theorem does not apply to the space $V$, that is, the normalized volume of the Newton polytope of a generic polynomial from $V$ is bigger than the number of solutions of a generic system $f_1=f_2=f_3=0$ with $f_i\in V$.
\end{example}

To assign the {\em Newton--Okounkov convex body} to $V$ we need an extra ingredient.
Choose a translation-invariant total order on the lattice $\Z^n$ (e.g., we can take the lexicographic order).
Consider a map
$$v:\C(x_1,\ldots,x_n)\setminus\{0\}\to\Z^n,$$
that behaves like the lowest order term of a polynomial, namely: $v(f+g)\ge \min\{v(f),v(g)\}$ and $v(fg)=v(f)+v(g)$ for all nonzero $f,g$.
Recall that maps with such properties are called {\em valuations}.
A straightforward construction of valuations is shown in Example \ref{e.flag2} below.

%One way to define $v$ is to compute the order of zeroes or poles of $f\in\C(x_1,\ldots,x_n)$ along coordinate hypersurfaces $\{x_i=0\}$ for $i=1$,\ldots, $n$.

\begin{defin}
The {\em Newton--Okounkov convex body} $\Delta_v(V)$ is the closure
of the convex hull of the set
$$\bigcup_{k=1}^\infty\left\{\frac{v(f)}{k} \ |\ f\in V^k \right\}\subset\R^n.$$
By $V^k$ we denote the subspace spanned by the $k$-th powers of the functions from $V$.
\end{defin}

Different valuations might yield different Newton--Okounkov convex bodies.
An important application of Newton--Okounkov bodies is the following analog of Kouchnirenko theorem.
Recall that by $X_0\subset\C^n$ we denoted an open dense subset where all functions from $V$ are regular (that is, do not have poles).
\begin{thm}\cite{KK,LM}\label{t.NO}
If $V$ is sufficiently big, then a generic system $f_1=\ldots=f_n=0$ with $f_i\in V$ has
$n!\Volume(\Delta_v(V))$ solutions in $X_0$.
\end{thm}
In particular, it follows that all Newton--Okounkov convex bodies for $V$ have the same volume.
For more details (in particular, for the precise meaning of ``sufficiently big'') we refer the reader to \cite[Theorem 4.9]{KK}.
\begin{example} \label{e.flag2} Let $V$ be the space from Example \ref{e.flag}.
Define a valuation $v$ by assigning to a polynomial $f\in\C[x_1,x_2,x_3]$ its lowest order term with respect to the
lexicographic ordering of monomials.
More precisely, we say that $x_1^{k_1}x_2^{k_2}x_3^{k_3}\succ x_1^{l_1}x_2^{l_2}x_3^{l_3}$ iff
there exists $j\le3$ such that $k_i=l_i$ for $i<j$ and $k_j>l_j$.
It is easy to check that $v(V)$ consists of $8$ lattice points  $(0,0,0)$, $(1,0,0)$, $(0,1,0)$, $(0,0,1)$, $(1,1,0)$,
$(1,0,1)$, $(0,1,1)$, $(0,2,0)$.
Their convex hull is depicted on Figure 1.
This is the FFLV polytope $FFLV(1,0,-1)$ for the adjoint representation of $SL_3$ (in this case, it happens to be unimodularly equivalent to the GZ polytope).
In particular, $FFLV(1,0,-1)\subset\Delta_v(V)$.
\end{example}
\begin{figure}
\begin{center}\includegraphics[width=10cm]{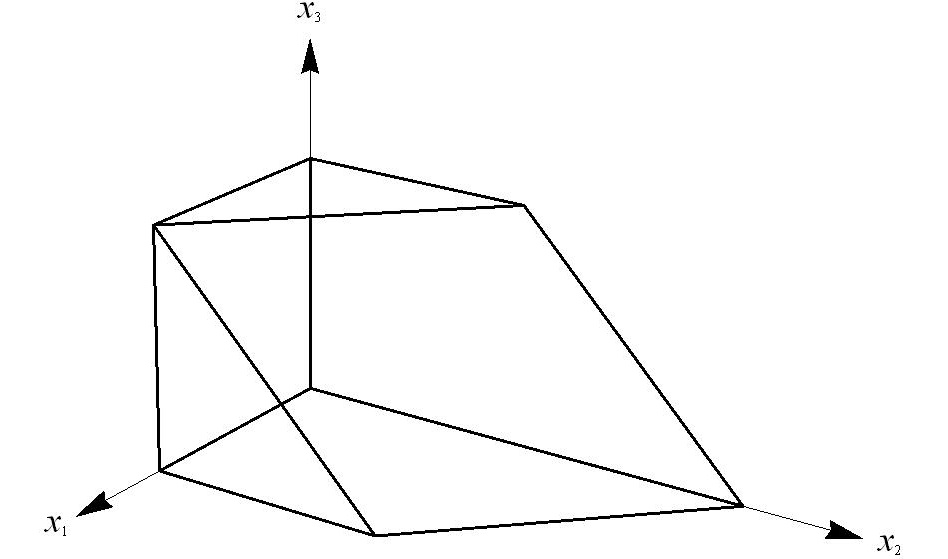}
\caption{}
\end{center}
\end{figure}

\subsection{Enumerative geometry}
\label{ss.sphere}
In this section, we give a brief introduction to enumerative geometry for the general mathematical audience.
Enumerative geometry motivated the study of Grassmannians, flag varieties and more general spherical varieties.
Recall two classical problems of enumerative geometry from the 19-th century.

\begin{problem}[Schubert]\label{p.Sch}
How many lines in a 3-space intersect four given lines in general position?
\end{problem}
We can identify lines in $\C\P^3$ with vector planes in $\C^4$, that is, a line can be viewed as a point on the
Grassmannian $G(2,4)$.
The condition that a line $l\in G(2,4)$ intersects a fixed line $l_1$ defines a hypersurface $H_1\subset G(2,4)$.
Hence, the problem reduces to computing the number of intersection points of four hypersurfaces in $G(2,4)$.
It is not hard to check that the hypersurface $H_1$ is just a hyperplane section of the Grassmannian under the Pl\"ucker
embedding $G(2,4)\hookrightarrow \P(\Lambda^2\C^4)\simeq\C\P^5$.
The image of the Grassmannian is a quadric in $\C\P^5$.
The number of intersection points of a quadric in $\C\P^5$ with
four hyperplanes in general position is equal to $2$ by the Bezout theorem.
Hence, the answer to the Schubert problem is $2$.

Schubert's problem can also be solved for real lines in $\R^3$ by elementary metods (for instance, by using two families
of lines on a hyperboloid of one sheet).
In this context, Schubert's problem was recently applied to experimental physics \cite{Phys}.

\begin{problem}[Steiner]\label{p.St}
How many smooth conics are tangent to five given conics?
\end{problem}
Similarly to the Schubert problem, we can identify conics with points in $\C\P^5$, namely, the conic given by an equation
$ax^2+bxy+cy^2+dxz+eyz+fz^2=0$ corresponds to the point $(a:b:c:d:e:f)\in\C\P^5$.
Smooth conics form an open subset $C\subset\C\P^5$ (the complement $\C\P^5\setminus C$ is the zero set of the discriminant).
The condition that a conic is tangent to a given conic defines a hypersurface in $\C\P^5$ of degree $6$.
Using Bezout theorem in $\C\P^5$ one might guess (as Jacob Steiner himself did) that the answer to the Steiner problem is $6^5$.
However, the correct answer is much smaller.
This is similar to the difference between the Bezout and Kouchnirenko theorems: the former yields
extraneous solutions that have no enumerative meaning.
The correct answer was found by Michel Chasles who used (in modern terms) a {\em wonderful compactification} of $C$, namely, the space of complete conics.

Hermann Schubert developed a powerful general method (calculus of conditions) for solving problems of enumerative geometry such as Problems \ref{p.Sch}, \ref{p.St}.
In a sense, his method was based on an informal version of intersection theory.
The 15-th Hilbert problem asked for a rigorous foundation of Schubert calculus\footnote{Das Problem besteht darin, {\em diejenigen geometrischen Anzahlen
strenge und unter genauer Feststellung der Grenzen ihrer G\"ultigkeit
zu beweisen, die insbesondere Schubert auf Grund des sogenannten Princips
der speciellen Lage mittelst des von ihm ausgebildeten
Abz\"ahlungskalk\"uls bestimmt hat} (Hilbert).}.
In the first half  of the 20-th century, these foundations were developed both in the topological (cohomology rings) and algebraic (Chow rings) settings.
However, Schubert's version of intersection theory was formalized only in the 1980-s by Corrado De Concini and Claudio Procesi \cite{CP}.

In particular, many problems of enumerative geometry (including Problems \ref{p.Sch} and \ref{p.St}) reduce to computation of the self-intersection index of a hypersurface in homogeneous space $G/H$ where $G$ is a reductive group such as $SL_n(\C)$, $SO_n(\C)$ or $Sp_{2n}(\C)$.
In the toric case ($G=(\C^*)^n$), the Kouchnirenko theorem yields an explicit formula for the self-intersection index of a hypersurface $\{f=0\}$ where $f$ is a generic polynomial with a given Newton polytope.
In the reductive case, explicit formulas were obtained by Boris Kazarnovskii (case of $(G\times G)/G^{\rm diag}$) and Michel Brion (general case) \cite{Kaz,Br89}.
Though the Brion--Kazarnovskii formula was originally stated in different terms, it can be reformulated using Newton--Okounkov polytopes \cite{KK2}.

\begin{example} \label{e.flag3} We now place Example \ref{e.flag} into the context of enumerative geometry.
Let $X=\{(V^1\subset V^2\subset \C^3) \ | \ \dim V^i=i \}$ be the variety of complete flags in $\C^3$.
This is a homogeneous space under the action of $SL_3(\C)$, namely, $X=SL_3(\C)/B$ where $B$ is the subgroup of upper-triangular matrices.
It is easy to check that $B$ acts on $X$ with an open dense orbit $U^-B/B\simeq U^-$.

We say that two flags $V^1\subset V^2$ and $W^1\subset W^2$ in $\C^3$ are not in general position if either $V^1\subset W^2$ or $W^1\subset V^2$.
How many flags in $\C^3$ are not in general position with three given flags?
By taking projectivizations of subspaces $V^1\subset V^2\subset\C^3$ we can regard a flag as $a\in l\subset\C\P^2$, where $a=\P(V^1)$ is a point and $l=\P(V^2)$ is a line on the projective plane.
Hence, we can reduce the question to the following elementary problem.
\begin{problem}[High school geometry]
There is a triangle $ABC$ on the plane.
Points $A'$, $B'$, $C'$ lie on the lines $BC$, $AC$ and $AB$, respectively.
Find all configurations $(X, YZ)$  (where a point $X$ lies on a line $YZ$) such that
$(X, YZ)$ is not in general position with the configurations $(A', BC)$, $(B', AC)$ and $(C', AB)$.
\end{problem}
It is easy to show that there are $6$ such configurations.

On the other hand, the same answer can be found using the simplest projective embedding of $X$:
$$p:X\hookrightarrow\P(\C^3)\times \P(\Lambda^2\C^3)\stackrel{\mbox{\tiny Segre}}{\hookrightarrow}\P({\rm End}(\C^3)); \quad
p:(V^1, V^2)\mapsto V^1\times V^2\mapsto V^1\otimes\Lambda^2 V^2,$$
and counting the number of intersection points of $p(X)$ with 3 generic hyperplanes in $\C\P^8$ (that is, the {\em degree} of $p(X)$).
Restricting the map $p$ to the open dense $B$-orbit $U^-\subset X$ we get that the latter problem reduces to the problem from Example \ref{e.flag}.
In particular, we can show that the inclusion $FFLV(1,0,-1)\subset\Delta_v(V)$ is an equality.
Indeed, by Theorem \ref{t.NO} the volume of $\Delta_v(V)$ times $3!$ is equal to the degree of $p(X)$, that is, to $6$.
Hence, the volume of $\Delta_v(V)$ is equal to $1$.
Since the volume of $FFLV(1,0,-1)$ is also equal to $1$, the inclusion $FFLV(1,0,-1)\subset\Delta_v(V)$ of convex polytopes implies the exact equality.
\end{example}

\section{GZ patterns and FFLV polytopes}\label{s.GZ}

In this section, we recall the definitions of GZ patterns in types $A$, $B$, $C$, $D$ and FFLV
polytopes in types $A$ and $C$.
Let $\lambda=(\lambda_1,\ldots,\lambda_n)$ denote a non-increasing collection of integers.
In what follows, we regard $\l$ as a dominant weight of a classical group.
GZ polytopes for classical groups $G$ were constructed using representation theory, namely, lattice points in the polytope $GZ(\lambda)$ parameterize the vectors of the GZ basis in the irreducible representation $V_\l$ of $G$ with the highest weight $\l$ (see \cite{M} for a survey on GZ bases).
Lattice points in FFLV polytopes  $FFLV(\l)$ parameterize a different basis in the same representation (see \cite{FFL,FFL2}).
In particular, $GZ(\l)$ and $FFLV(\l)$ have the same Ehrhart and volume polynomials.

\subsection{GZ patterns}\label{ss.GZ}
\subsubsection{Type A}
We now regard $\l$ as a dominant weight of $SL_n$.
In convex geometric terms, the GZ polytope $GZ(\l)\subset\R^d$, where $d:=\frac{n(n-1)}{2}$, is defined as the set of all points
$(u^1_1,u^1_2,\ldots, u^1_{n-1};u^2_1,\ldots,u^2_{n-2};\ldots; u^{n-1}_1)\in\R^d$ that satisfy the following interlacing inequalities:
$$
\begin{array}{cccccccccc}
\l_1&       & \l_2    &         &\l_3          &    &\ldots    & &       &\l_n   \\
    &\boxed{u^1_1}&         &\boxed{u^1_2}  &         & \ldots   &       &  &\boxed{u^1_{n-1}}&       \\
    &       &\boxed{u^2_1} &       &  \ldots &   &        &\boxed{u^2_{n-2}}&         &       \\
    &       &       &  \ddots   & &  \ddots   &      &         &         &       \\
    &       &       &  &\boxed{u^{n-2}_1}&     &  \boxed{u^{n-2}_2} &        &         &       \\
    &       &         &    &     &\boxed{u^{n-1}_1}&   &              &         &       \\
\end{array}
\eqno(GZ_A)$$
where the notation
$$
 \begin{array}{ccc}
  a &  &b \\
   & c &
 \end{array}
 $$
means $a\ge c\ge b$ (the table encodes $2d$ inequalities).

\subsubsection{Types B and C}
Let $\l$ be a dominant weight of $Sp_{2n}(\C)$, that is, all $\l_i$ are non negative.
Put $d=n^2$.
Denote coordinates in $\R^d$ by $(x^1_1,\ldots, x^1_n; y^1_1,\ldots, y^1_{n-1};\ldots; x_1^{n-1}, x^{n-1}_2, y^{n-1}_1; x^n_1)$.
For every $\l$, define the {\em symplectic GZ polytope} $SGZ(\l)\subset\R^d$ for $Sp_{2n}(\C)$ by the following interlacing inequalities:
$$
\begin{array}{ccccccccccc}
\l_1&       & \l_2  &      &\l_3   &       & \ldots   & \l_n    &         &0      &\\
    &\boxed{x^1_1}  &       &\boxed {x^1_2} &       &\ldots &          &         &\boxed{x^1_{n}}  &       &\\
    &       &\boxed{y^1_1}  &      &\boxed{y^1_2}  &       & \ldots   &\boxed{y^1_{n-1}}&         & 0     &\\
    &       &       &\boxed{x^2_1} &       &\ldots &          &         &\boxed{x^2_{n-1}}&       &\\
    &       &       &      & \boxed{y^2_1} &       & \ldots   &\boxed{y^2_{n-2}}&         & 0     &\\
    &       &       &      &       &\ddots &          & \vdots  &         &\vdots &\\
    &       &       &      &       &       & \boxed{x^{n-1}_1}&         &\boxed{x^{n-1}_2}&       &\\
    &       &       &      &       &       &          &\boxed{y^{n-1}_1}&         &     0 &\\
    &       &       &      &       &       &          &         & \boxed{x^{n}_1} &       &\\
\end{array}
\eqno{GZ_C}$$
Again, every coordinate in this table is bounded from above by its upper left neighbor and bounded from below by its upper right neighbor (the table encodes $2d$ inequalities).
Roughly speaking, $SGZ(\l)$ is the polytope defined using half of the GZ pattern $(GZ_A)$ for $SL_{2n}(\C)$.
%\begin{example}\label{e.SGZ}
%The polytope $SGZ_\l\subset\R^4$ for $Sp_4(\C)$ is given by 8 inequalities:
%$$\l_1\ge x^1_1\ge \l_2; \quad \l_2\ge x^1_2\ge 0; \quad  x^1_1\ge y^1_1\ge x^1_2; \quad y^1_1\ge x^2_1\ge 0.$$
%It is not hard to compute the volume polynomial of $SGZ_\l$:
%$$vol_{SGZ}(\l_1,\l_2)=\frac16\l_1\l_2(\l_1-\l_2)(\l_1+\l_2).$$
%This volume times $4!$ is equal to the degree $\deg_\l (Sp_4(\C)/B)$ of the isotropic flag variety.
%\end{example}

To define the GZ polytope in type $B$ (that is, for $G=SO_{2n+1}(\C)$) we use the same pattern and inequalities but choose a bigger lattice $L\subset\R^d$ so that the standard lattice $\Z^d\subset L$ has index $2$ in $L$ (see \cite{BZ} for more details).

\subsubsection{Type D}
Let $\l$ be a dominant weight of $SO_{2n}(\C)$.
Put $d=n(n-1)$.
Denote coordinates in $\R^d$ by $(y^1_1,\ldots, y^1_{n-1};\ldots; x_1^{n-1}, x^{n-1}_2, y^{n-1}_1; x^n_1)$.
For every $\l$, define the {\em even orthogonal GZ polytope} $OGZ(\l)\subset\R^d$ for $SO_{2n}(\C)$ using the following table:
$$
\begin{array}{cccccccccc}
\l_1  &       &\l_2 &       &\ldots &          &         &\l_n  &       &\\
      &\boxed{y^1_1}  &      &\boxed{y^1_2}  &       & \ldots   &\boxed{y^1_{n-1}}&         &     &\\
      &       &\boxed{x^2_1} &       &\ldots &          &         &\boxed{x^2_{n-1}}&       &\\
      &       &      & \boxed{y^2_1} &       & \ldots   &\boxed{y^2_{n-2}}&         &      &\\      &       &      &       &\ddots &          & \vdots  &         &\vdots &\\
      &       &      &       &       & \boxed{x^{n-1}_1}&         &\boxed{x^{n-1}_2}&       &\\
      &       &      &       &       &          &\boxed{y^{n-1}_1}&         &      &\\
      &       &      &       &       &          &         & \boxed{x^{n}_1} &       &\\
\end{array}
\eqno{GZ_D}$$
Again, every coordinate in this table is bounded from above by its upper left neighbor and bounded from below by its upper right neighbor.
There are also extra inequalities for every $i=1$,\ldots ,$n-2$:
$$x^i_{n-i}+x^i_{n+1-i}+x^{i+1}_{n-i}\ge y^i_{n-i}; \quad x^{i+1}_{n-i-1}+x^i_{n+1-i}+x^{i+1}_{n-i}\ge y^i_{n-i},$$
and inequality $x^{n-1}_{1}+x^{n-1}_{2}+x^{n}_{1}\ge y^{n-1}_{1}$ (see \cite{BZ} for more details).

In what follows, we will use not GZ polytopes themselves but the GZ tables.
\begin{remark}\label{r.shape}
If we rotate GZ tables in types $A$, $B$/$C$ and $D$ by $\frac{3\pi}4$ clockwise we will get the following tables:
$$
A \begin{array}{c}
\begin{array}[t]{|c|}
\hline
\ \\
\hline
\ \\
\hline
\ \\
\hline
\end{array}
\begin{array}[t]{|c|}
\hline
\ \\
\hline
\ \\
\hline
\end{array}
\begin{array}[t]{|c|}
\hline
\ \\
\hline
\end{array}
\end{array}, \quad
B/C \begin{array}{c}
\begin{array}[t]{|c|}
\hline
\ \\
\hline
\end{array}
\begin{array}[t]{|c|}
\hline
\ \\
\hline
\ \\
\hline
\end{array}
\begin{array}[t]{|c|}
\hline
\ \\
\hline
\ \\
\hline
\ \\
\hline
\end{array}
\begin{array}[t]{|c|}
\hline
\ \\
\hline
\ \\
\hline
\end{array}
\begin{array}[t]{|c|}
\hline
\ \\
\hline
\end{array}
\end{array},\quad
D \begin{array}{c}
\begin{array}[t]{|c|}
\hline
\ \\
\hline
\end{array}
\begin{array}[t]{|c|}
\hline
\ \\
\hline
\ \\
\hline
\end{array}
\begin{array}[t]{|c|}
\hline
\ \\
\hline
\ \\
\hline
\ \\
\hline
\end{array}
\begin{array}[t]{|c|}
\hline
\ \\
\hline
\ \\
\hline
\ \\
\hline
\end{array}
\begin{array}[t]{|c|}
\hline
\ \\
\hline
\ \\
\hline
\end{array}
\begin{array}[t]{|c|}
\hline
\ \\
\hline
\end{array}
\end{array}
$$
We will use this presentation of GZ tables in the proof of Theorem \ref{t.C}.
\end{remark}

\subsection{FFLV polytopes} \label{ss.FFLV}
\subsubsection{Type A}
For every dominant weight $\l$ of $SL_n(\C)$, we now define the FFLV polytope $FFLV(\l)$.
Put $d:=\frac{n(n-1)}2$.
Label coordinates in $\R^d$ by
$(u^1_{n-1};u^2_{n-2},u^1_{n-2};\ldots;u^{n-1}_1,u^{n-2}_{1},\ldots,u^{1}_1)$.
and organize them using the GZ table $(GZ_A)$.
%$$
%\begin{array}{cccccccccc}
%\l_1&       & \l_2    &         &\l_3          &    &\ldots    & &       &\l_n   \\
%    &u^1_1&         &u^1_2  &         & \ldots   &       &  &u^1_{n-1}&       \\
%    &       &u^2_1 &       &  \ldots &   &        &u^2_{n-2}&         &       \\
%    &       &       &  \ddots   & &  \ddots   &      &         &         &       \\
%    &       &       &  &u^{n-2}_1&     &  u^{n-2}_2 &        &         &       \\
%    &       &         &    &     &u^{n-1}_1&   &              &         &       \\
%\end{array}
%\eqno{(FFLV)}$$
The polytope $FFLV(\l)$ in type $A$ is defined by inequalities
$u^l_m\ge 0$ and
$$\sum_{(l,m)\in D}u^l_m\le \l_i-\l_j$$
for all Dyck paths $D$ going from $\l_i$ to $\l_j$ in table $(FFLV)$  where $1\le i<j\le n$.
A {\em Dyck path} is a broken line whose segments either connect $u^i_j$ with $u^{i+1}_j$ or
connect $u^i_j$ with $u^i_{j+1}$.
%See \cite{FFL} for a connection between lattice points in $FFLV(\l)$ and vectors in a special basis in the irreducible representation $V_\l$.
Note that $FFLV(\l)$ only depends on the differences $(\l_1-\l_2)$,\ldots, $(\l_{n-1}-\l_n)$.
An example of FFLV polytope for $n=3$ and $\l=(1,0,-1)$ is depicted on Figure 1.
\subsubsection{Type C}
Similarly to type A case, we define FFLV polytopes in type $C$ using the corresponding GZ table $(GZ_C)$.
The only difference with type $A$ case is that we allow Dyck paths to end at one of the $0$ entry in the rightmost column of the table (see \cite{FFL2,ABS} for more details).

\section{Valuations on flag varieties}\label{s.main}
We now construct uniformly a valuation $v$ on flag varieties in types $A$, $B$, $C$ and $D$.
In types $A$ and $C$, we identify the corresponding Newton--Okounkov polytopes with FFLV polytopes.
In type $B_2$, we get a {\em symplectic DDO polytope} \cite[Section 4]{Ki16I}, which is not
combinatorially equivalent to either the FFLV or the GZ polytope in type $C_2$.
In type $D_3$, we get a polytope that is different from both GZ and FFLV polytopes in type $A_3$, however, the question of combinatorial equivalence is open.

Fix a complete flag of subspaces
$F^\bullet:=(F^1\subset F^2\subset\ldots\subset F^{n-1}\subset \C^n)$, and a basis $e_1$,\ldots, $e_n$ in $\C^n$ compatible with $F^\bullet$, that is, $F^i=\langle e_1,\ldots,  e_i\rangle$.
Define a non-degenerate symmetric bilinear $(\cdot,\cdot)$ form on $\C^n$ as follows:
$$ (e_i,e_j)=\left\{
\begin{array}{cc} 1 &\mbox{ if } i+j=n;\\
0 & \mbox{ otherwise }.
\end{array} \right.$$
Similarly, we define a non-degenerate skew symmetric form $\om(\cdot,\cdot)$ for even $n$.
For $i<j$, put 
$$\om(e_i,e_j)=-\om(e_j,e_i)=\left\{
\begin{array}{cc} 1 &\mbox{ if } i+j=n;\\
0 & \mbox{ if } i+j\ne n.
\end{array} \right.$$

Let $B\subset SL_n(\C)$ be a subgroup of upper triangular matrices with respect to the basis $e_1$,\ldots, $e_n$.
Recall that the complete flag variety $SL_n(\C)/B$ can be defined as the variety of complete flags of subspaces $M^\bullet=(\{0\}\subset V^1 \subset V^2\subset\ldots\subset V^{n-1}\subset\C^n)$.
Similarly, we regard $SO_n(\C)/B$ for any $n$ and $Sp_{n}/B$ for even $n$ as subvarieties of {\em orthogonal} and {\em isotropic} flags in $SL_n(\C)/B$ and $SL_{n}/B$, respectively.
A complete flag $M^\bullet$ in $\C^n$ is {\em orthogonal} if $V^i$ is orthogonal to to $V^{n-i}$ with respect to
$(\cdot,\cdot)$.
A complete flag $M^\bullet$ in $\C^{n}$ is called {\em isotropic} if the restriction of $\omega$ to $V^{\frac{n}{2}}$
is zero, and $V^{n-i}=\{v\in \C^{n} \ | \ \omega(v,u)=0 \mbox{ for all } u\in V^i\}$.
In particular, the flag $F^\bullet$ is orthogonal and isotropic by our choice of the forms $(\cdot,\cdot)$ and $\omega$.

Recall that if $G$ is a connected complex semisimple group (e.g., a classical group), then
the Picard group of the complete flag variety $G/B$ can be identified with the weight lattice of  $G$ \cite[1.4.2]{B}.
In particular, there is a bijection between dominant weights $\l$ and globally generated line bundles $L_\l$.
Recall also that the space of global sections $H^0(G/B,L_\l)$ is isomorphic to $V_\l^*$ where
$V_\l$ is the irreducible representation of $G$ with the highest weight $\l$.
Let $v_\l\in V_\l$ be a highest weight vector, i.e., the line $\langle v_\l\rangle\subset V_\l$ is $B$-invariant.
There is a well-defined map
$$p_\l:G/B\to\P(V_\l), \quad gB\mapsto\langle v_\l\rangle\subset \P(V_\l).$$
For instance, if $G=SL_3$ and $\l=(1,0,-1)$, then $p_\l$ coincides with the map $p$ of Example \ref{e.flag3}.
Similarly to Example \ref{e.flag} we may identify $V_\l^*$ with a subspace of $\C(G/B)$.
This amounts to fixing a global section $s_0\in H^0(G/B,L_\l)$ and identifying $s\in H^0(G/B,L_\l)$ with $\frac{s}{s_0}\in \C(G/B)$.
Denote by $\Delta_v(G/B,L_\l)\subset\R^d$ the Newton--Okounkov convex body corresponding to $G/B$,
$L_\l$ and $v$ (we denote by $d$ the dimension of $G/B$).
In what follows, we use that the normalized volume of $\Delta_v(G/B,L_\l)$
is equal by Theorem \ref{t.NO} to the degree of $p_\l(G/B)\subset\P(V_\l)$.
The latter is equal to the volume of $GZ(\l)$ and $FFLV(\l)$ by the Hilbert's theorem.

\subsection{Type A} Let $G=SL_n(\C)$.
Put $d=\frac{n(n-1)}{2}$.
Recall that the open Schubert cell $X^\circ$ with respect to $F^\bullet$ is defined as the set of all flags $M^\bullet$ that are in general position with the standard flag $F^\bullet$, i.e., all  intersections $M^i\cap F^j$ are transverse.
We can identify the open Schubert cell  $X^\circ\subset G/B$  with an affine space $\C^d$ by choosing for every flag $M^\bullet$  a basis  $v_1$,\ldots, $v_n$ in $\C^n$ of the form:
$$v_1=e_n+x^{n-1}_1 e_{n-1}+\ldots+ x^1_1 e_1, $$
$$v_2=e_{n-1}+x^{n-2}_2 e_{n-2}+\ldots+ x^1_2 e_1, \quad \ldots \quad, v_{n-1}=e_2+x^1_{n-1}e_1, \quad v_n=e_n,$$
so that $M^i=\langle v_1,\ldots,  v_i\rangle$.
Such a basis is unique, hence, the coefficients $(x^i_j)_{i+j< n}$ are coordinates on the open cell.
In other words, every flag $M^\bullet\in X^\circ$ gets identified with a triangular matrix:
$$\begin{pmatrix}
   x^1_1 & x^1_2& \ldots & x^1_{n-1} & 1\\
   x^2_1 & x^2_2   & \ldots & 1 & 0\\
   \vdots & \vdots   &  &  & \vdots\\
   x^{n-1}_1 & 1 & \ldots & 0 & 0\\
    1        & 0 &        & 0 & 0\\
  \end{pmatrix}.\eqno(*)
$$
We order the coefficients $(x^i_j)_{i+j< n}$ of this matrix by starting from column $(n-1)$ and going from top to bottom in every column and from right to left along columns.
More precisely, put $(y_1,\ldots,y_d):=(x^1_{n-1};x^1_{n-2},x^2_{n-2};\ldots;x^{1}_1,x^{2}_{1},\ldots,x^{n-1}_1)$.
For instance, if $n=4$ we get the ordering:
$$\begin{pmatrix}
   y_4 & y_2 & y_1 & 1\\
   y_5 & y_3 & 1  & 0\\
   y_6 &  1  & 0  & 0\\
   1   &  0  & 0  & 0\\
  \end{pmatrix}.
$$
We fix the lexicographic ordering on monomials in coordinates $y_1$, \ldots, $y_d$ so that $y_1\succ y_2\succ\ldots\succ y_d$.
By the lexicographic ordering we mean that  $y_1^{k_1}\cdots y^{k_d}\succ y^{l_1}\cdots y^{l_d}$ iff
there exists $j\le d$ such that $k_i=l_i$ for $i<j$ and $k_j>l_j$.

\begin{remark}\label{r.flag}
In \cite[Section 2.2]{K17}, there is a geometric construction of coordinates compatible with the flag of translated Schubert subvarieties:
$$w_0X_{\id}\subset w_0w_{d-1}^{-1}X_{w_{d-1}}\subset w_0w_{d-2}^{-1}X_{w_{d-2}}
\subset\ldots\subset w_0w_{1}^{-1}X_{w_{1}}\subset SL_n/B,$$
for $\w0=(s_1)(s_2s_1)(s_3s_2s_1)\ldots(s_{n-1}\ldots s_1)$.
Here $w_k$ denotes the $k$-th terminal subword of $\w0$, that is, $w_{d-1}=s_1$, $w_{d-2}=s_2s_1$ and so on.
It is not hard to check that coordinates
$(y_1,\ldots,y_d)$
are also compatible with the same flag, i.e.,
$w_0w_{k}^{-1}X_{w_k}\cap X^\circ=\{y_1=\ldots=y_k=0\}$.
\end{remark}

Let $v$ denote the lowest order term valuation on $\C(G/B)$, that is, if $y_1^{k_1}\cdots y_d^{k_d}$
is the lowest order term of a polynomial  $f\in\C(G/B)$ then $v(f):=(k_1,\ldots,k_d)\in\Z^d$.
For the ratio $\frac{f}{g}$ of two polynomials we put $v(\frac{f}{g}):=v(f)-v(g)$.
Let $L_\l$ be the line bundle on $G/B$ corresponding to a dominant weight
$\l:=(\l_1,\ldots,\l_n)\in\Z^n$ of $G$.

\begin{thm}{\cite[Theorem 2.1]{K17}}\label{t.A}
In type $A$, the Newton--Okounkov convex body $\Delta_v(G/B,L_\l)$ coincides with the
FFLV polytope
$FFLV(\l)$.
\end{thm}

For instance, the computation of the polytope $\Delta_v(SL_n/B,L_\l)$ for $n=3$ and $\l=(1,0,-1)$ is illustrated in Examples \ref{e.flag}, \ref{e.flag2}, \ref{e.flag3}.
Using Remark \ref{r.flag} we could deduce Theorem \ref{t.A} directly from  \cite[Theorem 2.1]{K17}.
Below we give another proof that works simultaneously for types $A$ and $C$.

\subsection{Type $C$} Let $n=2r$ be even, and $G=Sp_n(\C)$.
Put $d=r^2$.
We define the open Schubert cell $X^\circ$ with respect to $F^\bullet$ as the set of all  isotropic flags $M^\bullet$
that are in general position with the standard flag $F^\bullet$.
Again, we can identify the open Schubert cell  $X^\circ\subset G/B$  with an affine space $\C^d$ using matrix $(*)$.
Since $M^\bullet$ is isotropic the coefficients $(x^i_j)_{i+j< n}$ are no longer independent variables.
It is not hard to check that exactly $d$ coefficients, namely, $(x^i_j)_{i+j< n, i\le j}$ are independent.
Again, we order the coordinates by starting from column $(n-1)$ and going from top to bottom in every column and from right to left along columns.
That is, put $(y_1,\ldots,y_d):=(x^1_{n-1};x^1_{n-2},x^2_{n-2};\ldots;x^{1}_r,x^{2}_{r},\ldots,x^r_r; \ldots; x^1_2,x^2_2; x^1_1)$.

It is easy to check that every $x^i_j$ for $i>j$ can be expressed as a polynomial in coordinates $y_1$,\ldots, $y_d$ with the lowest order term $x^j_i$.
In particular, there is a table inside the matrix $(*)$ whose coefficients
are coordinates on the Schubert cell $X^\circ$.
Here is an example for $n=6$ (coefficients inside the table are boxed):
%$$\begin{pmatrix}
%\boxed{x^1_1} & \boxed{x^1_2}   &\boxed{x^1_3} &\ldots        &   &\boxed{x^1_{n-2}} & \boxed{x^1_{n-1}} & 1\\
%x^2_1         & \boxed{x^2_2}   &\boxed{x^2_3} &\ldots        &   &\boxed{x^2_{n-2}} & 1                 & 0\\
%\vdots        & \vdots          &\ddots        &              &   &                  &                   & 0\\
%x^{k}_1       & \ldots          &x^k_{k-1}     & \boxed{x^k_k}&1  & 0                & \ldots            & 0\\
%\end{pmatrix}.
%$$
$$\begin{pmatrix}
\boxed{x^1_1} & \boxed{x^1_2}   &\boxed{x^1_3} & \boxed{x^1_4}& \boxed{x^1_{5}}   & 1\\
x^2_1         & \boxed{x^2_2}   &\boxed{x^2_3} & \boxed{x^2_4}&  1                & 0\\
x^3_1         & x^3_2           &\boxed{x^3_3} & 1            &  0                & 0\\
x^4_1         & x^4_2           &1             & 0            &  0                & 0\\
x^5_1         & 1               &0             & 0            &  0                & 0\\
1             & 0               &0             & 0            &  0                & 0\\
\end{pmatrix}.
$$
Note that the table is shaped exactly as the GZ pattern in type $C$ rotated by $\frac{3\pi}4$ clockwise (see Remark \ref{r.shape}).

Similarly to the type $A$ case, let $v$ be the lowest term valuation on $\C(G/B)$ associated with this ordering.
Let $L_\l$ be the line bundle on $G/B$ corresponding to a dominant weight
$\l:=(\l_1,\ldots,\l_r)\in\Z^r$ of $G$.
As before, denote by $\Delta_v(G,L_\l)\subset\R^d$
the Newton--Okounkov convex body corresponding to $G/B$,
$L_\l$ and $v$.

\begin{thm}\label{t.C}
In type $C$, the Newton--Okounkov convex body $\Delta_v(G/B,L_\l)$ coincides with the
FFLV polytope $FFLV(\l)$ in type $C$.
\end{thm}

\begin{proof}
We will provide a uniform proof for types $A$ and $C$.
Note that in both types the irreducible representation corresponding to the fundamental weight $\om_k$
is contained in the $k$-th exterior power of the tautological representation $G\subset GL_n(\C)$ \cite[Exercise 15.14, Theorem 17.5]{FH}.
Hence, the space $H^0(G/B,L_{\om_k})=V_{\om_k}^*$ is spanned by the restrictions of
Pl\"ucker coordinates of the Grassmannian $G(k,n)\subset \P(\Lambda^k\C^n)$ to $X^\circ$.
More precisely, there is a map $X^\circ\subset G/B\to G(k,n)\to \P(\Lambda^k\C^n)\supset\P(V_{\om_k})$, which allows us to identify
$V_{\om_k}^*$ with a subspace of $\C(X^\circ)=\C(G/B)$ spanned by certain $k\times k$ minors of matrix $(*)$.
Namely, we take all $k\times k$ minors of the $k\times n$ submatrix  of $(*)$ formed by the first $k$ rows.
This is equivalent to taking all minors of $k\times (n-k)$ submatrix $A_{k,n-k}$ of $(*)$ with coefficients $x^i_j$ where $i\le k$ and $j\le (n-k)$.

It follows easily from the definition of the valuation $v$ that the lowest order term in any minor of matrix $A_{k,n-k}$
is the diagonal term.
Hence, $v(V_{\om_k}^*)$ consists precisely of those points with coordinates $(u^i_j)$ in $\R^d$ such that $u^i_j=0,1$ and
two nonzero $u^i_j$ never lie on the same Dyck path.
Hence, the convex hull of $v(V_{\om_k}^*)$ coincides with the FFLV polytope $FFLV(\om_i)$.
We get the inclusion $FFLV(\om_i)\subset \Delta_v(G/B,L_{\om_i})$.
By the superadditivity of Newton--Okounkov convex bodies \cite[Proposition 2.32]{KK} we also have that if
$\l=\sum m_i\om_i$ then
$$\sum m_i\Delta_v(G/B,L_{\om_i})\subset \Delta_v(G/B,L_\l),$$
where the addition in the left hand side is Minkowski sum.
By definition $FFLV(\l)=\sum m_i FFLV(\om_i)$.
Hence, we get inclusion
$$FFLV(\l)\subset \Delta_v(G/B,L_\l).$$
This inclusion is equality because both convex bodies have the same volume.
\end{proof}
\begin{remark}
The proof relies on the fact that the volume of $FFLV(\l)$ is equal to the degree of $p_\l(G/B)\subset\P(V_\l)$. 
In types $A$ and $C$, this fact has both representation theoretic \cite{FFL,FFL2} and combinatorial proofs \cite{ABS}.
In type $A$, there is also a convex geometric proof \cite[Section 4]{K17}. 
It would be interesting to check whether this proof extends to type $C$. 
\end{remark}
Similarly to the type $A$ case, the valuation $v$ in type $C$ can be defined using a flag of translated Schubert subvarieties, however, they no longer correspond to terminal subwords of any  decomposition of the longest element in the Weyl group of $G$.
For instance, if $n=4$ we get subvarieties corresponding to elements $s_2s_1s_2$, $s_1s_2$ and $s_1$ of the Weyl group.

\subsection{Type $B$} Let $n=2r+1$ be odd, and $G=SO_n(\C)$.
Put $d=r^2$.
We define the open Schubert cell $X^\circ$ with respect to $F^\bullet$ as the set of all orthogonal flags $M^\bullet$ that are in general position with the standard flag $F^\bullet$.
Again, there is a table (shaped as the GZ pattern in type $C$) inside the matrix $(*)$ whose coefficients are coordinates on the Schubert cell $X^\circ$.
Here is an example for $n=5$ (coefficients inside the table are boxed):
$$\begin{pmatrix}
x^1_1   &\boxed{x^1_2} & \boxed{x^1_3}& \boxed{x^1_{4}}   & 1\\
x^2_1   &x^2_2         & \boxed{x^2_3}&  1                & 0\\
x^3_1   &x^3_2         & 1            &  0                & 0\\
x^4_1   &1             & 0            &  0                & 0\\
1       &0             & 0            &  0                & 0\\
\end{pmatrix}.
$$

Put $(y_1,\ldots,y_d):=(x^1_{n-1};x^1_{n-2},x^2_{n-2};\ldots;x^{1}_{r+1},x^{2}_{r+1},\ldots,x^r_{r+1}; \ldots; x^1_3,x^2_3; x^1_2)$.
As before, let $v$ be the lowest term valuation on $\C(G/B)$ associated with the ordering $y_1\succ\ldots\succ y_d$.
It is easy to check that every $x^i_j$ for $i>j$ can be expressed as a polynomial in coordinates
$y_1$,\ldots, $y_d$ with the lowest order term $x^j_i$, while $x^i_i$ is a polynomial with the lowest order term $(x^i_{r+1})^2$.

While we may still use Pl\"ucker coordinates to compute $\Delta_v(G/B,L_{\om_k})$ it is no longer true that the lowest
order term in any minor of matrix $A_{k,n-k}$ is the diagonal term
(because the diagonal coefficients $x^i_i$ might contribute higher order terms).
In particular, the defining inequalities for the convex hull $P_k$ of $v(H^0(G/B,L_{\om_k}))$ will be more intricate.
Still, they can be described by generalizing the notion of Dyck paths.
It would be interesting to compare these inequalities with those of \cite{BK} (type $B_3$), see also \cite{Ma}.
To check whether the polytope $P_\l:=\sum m_iP_i$ coincides with the convex body $\Delta_v(G/B,L_{\l})$ for
$\l=\sum m_i\om_i$  we have to compare their volumes.
For instance, one could try to construct a volume preserving piecewise linear map between
$P_\l$ and the corresponding GZ-polytope in type $B$ extending the construction of \cite[Section 4.2]{K17}.

For $B_2$, it is easy to check using Pl\"ucker coordinates that the convex hull $P_\l$ of $v(H^0(G/B,L_{\l}))\subset\R^4$
for $\l=\l_1\om_1+\l_2\om_2$ contains the Minkowski sum $\l_1 P_1+\l_2 P_2$, where
$P_1$ is the 3-dimensional simplex with the vertices $(0,0,0,0)$, $(1,0,0,0)$, $(0,2,0,0)$, $(0,0,0,1)$,
and $P_2$ is the 3-dimensional simplex with the vertices $(0,0,0,0)$, $(0,1,0,0)$, $(0,0,1,0)$, $(0,0,0,1)$.
Hence, $P_\l$ is identical to the Newton--Okounkov polytope computed in \cite[Proposition 4.1]{Ki16I} (up to relabeling of coordinates).
Denote coordinates in $\R^4$ by $(u_1,u_2,u_3,u_4)$. Then $P_\l$ is given by inequalities:
$$0\le u_1, u_2, u_3, u_4; \quad  u_1 \le \l_1; \quad  u_3\le \l_2; \quad  2u_1+u_2+2(u_3+u_4) \le 2(\l_1+\l_2); $$ $$2u_1+u_2+u_3+u_4 \le 2\l_1+\l_2.$$
In particular, its volume coincides with the degree of $p_\l(G/B)\subset\P(V_\l)$.
Hence, $P_\l=\Delta_v(G/B,L_{\l})$.
Note that $P_\l$ is not combinatorially equivalent to the FFLV polytope in type $C_2$
(see \cite[Section 2.4]{K17}).

\subsection{Type $D$}
Let $n=2r$ be even, and $G=SO_n(\C)$.
Put $d=r(r-1)$.
There is a table (shaped as the GZ pattern in type $D$) inside the matrix $(*)$ whose coefficients are coordinates on the Schubert cell $X^\circ$.
Here is an example for $n=6$ (coefficients inside the table are boxed):
$$\begin{pmatrix}
x^1_1   &\boxed{x^1_2} & \boxed{x^1_3}& \boxed{x^1_{4}}   & \boxed{x^1_{5}}&1\\
x^2_1   &x^2_2         & \boxed{x^2_3}& \boxed{x^2_{4}}   & 1              &0\\
x^3_1   &x^3_2         & 0            &  1                & 0              &0\\
x^4_1   &x^4_2         & 1            &  0                & 0              &0\\
x^5_1   &1             & 0            &  0                & 0              &0\\
1       &0             & 0            &  0                & 0              &0\\
\end{pmatrix}.
$$
Put $(y_1,\ldots,y_d):=(x^1_{n-1}$; $x^1_{n-2},x^2_{n-2}$;\ldots; $x^{1}_{r+1},x^{2}_{r+1},\ldots, x^{r-1}_{r+1}$; $x^{1}_{r},x^{2}_{r},\ldots,x^{r-1}_{r}$;\ldots; $x^1_3,x^2_3; x^1_2)$, and define $v$ as before.
It is easy to check that every $x^i_j$ for $i>j$ can be expressed as a polynomial in coordinates
$y_1$,\ldots, $y_d$ with the lowest order term $x^j_i$, while $x^i_i$ is a polynomial with the lowest order term $x^i_r x^i_{r+1}$.

For $D_3$, it is not hard to check using Pl\"ucker coordinates that the convex hull $P_\l$ of $v(H^0(G/B,L_{\l}))\subset\R^6$
for $\l=\l_1\om_1+\l_2\om_2+\l_3\om_3$ contains the Minkowski sum $\l_1 P_1+\l_2 P_2+\l_3P_3$, where
$P_1$ is the 4-dimensional polytope with the vertices $(0,0,0,0,0,0)$, $(1,0,0,0,0,0)$, $(0,1,0,0,0,0)$, $(0,0,0,1,0,0)$, $(0,0,0,0,0,1)$, $(0,1,0,1,0,0)$,
$P_2$ is the 3-dimensional simplex with the vertices $(0,0,0,0,0,0)$, $(0,1,0,0,0,0)$, $(0,0,1,0,0,0)$, $(0,0,0,0,0,1)$, and
$P_3$ is the 3-dimensional simplex with the vertices $(0,0,0,0,0,0)$, $(0,0,0,1,0,0)$, $(0,0,0,0,1,0)$, $(0,0,0,0,0,1)$.

By reordering coordinates, we can get that $P_1=FFLV(\wt\om_2)$, $P_2=FFLV(\wt\om_1)$, $P_3=FFLV(\wt\om_3)$ for fundamental weights $\wt\om_1$, $\wt\om_2$, $\wt\om_3$ of $SL_4(\C)$.
However, these reorderings do not agree for different fundamental weights so it is not clear whether $P_\l$ is unimodularly equivalent to $FFLV(\l)$ in type $A_3$ (or to other known polytopes).
To compare $P_\l$ with FFLV and GZ polytopes one might write down the inequalities that define $P_\l$ and use them to count the number of facets of $P_\l$.
It would also be interesting to compute the inequalities for $P_\l$ in the case of $D_4$ and compare them with those of
\cite{G}.

\end{document}